\documentclass{article}
\pdfoutput=1

\usepackage{amssymb}
\usepackage{amsmath}
\usepackage{amsthm}
\usepackage{a4wide}

\usepackage[english]{babel}

\usepackage{graphicx}

\newcommand{\ovline}{\overline}
\newcommand{\wtilde}{\widetilde}

\newcommand{\rarr}{\rightarrow}
\newcommand{\larr}{\leftarrow}
\newcommand{\xrarr}{\xrightarrow}
\newcommand{\lrarr}{\leftrightarrow}
\newcommand{\Rarr}{\Rightarrow}
\newcommand{\LRarr}{\Leftrightarrow}

\newcommand{\bsl}{\backslash}

\newcommand{\nullset}{\varnothing}

\newcommand{\lesq}{\leqslant}
\newcommand{\greq}{\geqslant}

\newcommand{\veps}{\varepsilon}

\newcommand{\intr}{\mathrm{Int}\,}
\newcommand{\diam}{\mathrm{diam}}

\newcommand{\inprod}[2]{{\langle} #1, #2 {\rangle}}
\newcommand{\inprodbig}[2]{\bigl\langle #1, #2 \bigr\rangle}

\newcommand{\forceparindent}{\hskip 1.5em}

\newtheorem{prop}{Proposition}
\newtheorem{lemma}{Lemma}
\newtheorem{theorem}{Theorem}

\newcounter{Remarknum}
\newenvironment{remark}{\refstepcounter{Remarknum}{\bf Remark \arabic{Remarknum}.}}{\par}

\newcounter{Examplnum}
\newenvironment{example}{\refstepcounter{Examplnum}{\bf Example \arabic{Examplnum}. $\vartriangleleft$}}{$\vartriangleright$\par}

\begin{document}

\title{On covering a ball by congruent subsets in normed spaces}

\author{Sergij V. Goncharov\thanks{Faculty of Mechanics and Mathematics, Oles Honchar Dnipro National University, 72 Gagarin Avenue, 49010 Dnipro, Ukraine.
\textit{E-mail: goncharov@mmf.dnulive.dp.ua}}}

\date{August 2017}

\maketitle

\begin{abstract}
We consider the covering of a ball in certain normed spaces by its congruent subsets
and show that if the finite number of sets is not greater than the dimensionality of the space,
then the centre of the ball either belongs to the interior of each set, or doesn't belong to the interior of any set.
We also provide some examples when it belongs to the interior of exactly one set.
These are the specific cases of the modified problem originally posed for dissection.
\end{abstract}

\centerline{\small \textit{MSC2010:} Pri 46B20, Sec 46C05, 52C17, 52C22}

\centerline{\small \textit{Keywords:} ball, normed space, Hilbert space, cover, subset, congruence, motion, centre, interior, LSB}

\makeatletter
\@starttoc{toc} 
\makeatother

\section*{Introduction}
\addcontentsline{toc}{section}{Introduction}

\forceparindent
In \cite[C6, p. 87]{croft1991} the problem, attributed to S.K. Stein, is posed:
``Whether it is possible to partition the unit circle into congruent pieces so that the center is in the interior of one of the pieces?''.
At present, for arbitrary number of pieces it is considered to be unsolved (\cite{mathoverflow17313}).
It can be generalized and varied in many ways, as stated in same place (\cite[p. 88]{croft1991}), not only dimensionality.

Some related questions were studied and answered more or less fully, ---
\cite{douwen1993}, \cite{edelstein1988}, \cite{haddley2016}, \cite{kiss2016}, \cite{richter2008}, \cite{wagon1983} to name a few.
A problem of this kind may depend greatly on the meaning of involved terms like ``piece'', ``partition'', ``congruence'':
do we allow the pieces to intersect at boundaries?
does congruence include reflection? should the piece be connected? measurable?
For example, it is shown in \cite{wagon1983} that the ball in $\mathbb{R}^m$ cannot be ``strictly'' dissected
into $n\in [2;m]$ topologically congruent pieces, to say nothing of the centre; see also \cite{waerden1949}, \cite[25.A.6, p. 599]{gleason1980}, \cite{edelstein1988}.

Hereinafter, we distinguish between 3 types of ``decomposition'' of the set $B$ (in particular, the~ball) into the congruent (sub)sets $\{ A_i \}_{i\in I}$,
so that $B = \bigcup\limits_{i \in I} A_i$ (cf. \cite[p. 79]{croft1991}, \cite[p. 49]{hertel2003}):

$\bullet$ {\it partition}: $\{ A_i \}$ are pairwise disjoint;

$\bullet$ {\it dissection}: interiors of $\{ A_i \}$ are pairwise disjoint;

$\bullet$ {\it covering} (or {\sl intra}{\it covering} to emphasize $A_i \subseteq B$): no additional constraints are {\sl required}.

These terms aren't ``standardized'', and may have quite different meaning in other works.

Any partition is a dissection, and any dissection is a covering. Therefore, the impossibility of covering satisfying certain additional conditions (e.g. relating to the centre)
implies that dissection and partition satisfying the same conditions are not possible as well.
However, when such covering exists, the corresponding dissection or partition may not exist.

Here we consider the ``decomposition'' of (intra)covering type, in certain specific cases, while the original problem almost surely belongs to dissection type;
and the majority of referenced papers, temporally ordered from \cite{waerden1949} to \cite{kiss2016}, deals with partition.

{\ }

{\sl The routinism of inference suggests few/some/most/all of the presented ``results'' to be well known, even if not claimed explicitly or publicly,
and the aim is rather to remove the delusion that there's no such well-knowness...}
There are works concerning the original centre-in-interior dissection problem, under ``natural'' (or ``physical'') assumptions
(such as the space being Euclidean, boundaries of parts being rectifiable, parts being connected): cf. \cite{haddley2016}, \cite{kanelbelov2002}, \cite{banakh2010}.

In our opinion, the most similar negative result relating to covering is obtained for pre-Hilbert spaces in \cite[Th. 1.1]{douwen1993}:
in spite of ``indivisibility'' and ``partitioning'' terms, it is actually a covering considered there, under the assumption that exactly one set contains the centre.
See Rem. \ref{remIneqCounterEx} below.

{\ }

We try to attain the generality by considering spaces and coverings with as few additional properties and constraints as possible.
Thereby few different interpretations of the problem (the ball is closed/open etc.) are aggregated.

{\ }

Hereinafter, we consider a normed linear space $X$ over the field of reals $\mathbb{R}$, $\theta$ is the zero of $X$.
Where we need the specific space such as $\mathbb{R}^m$, we will note it.

Completeness of $X$ is not assumed.

$\| x \|$ is the norm of $x\in X$, inducing the metric $\rho (x, y) = \| x - y \|$.

The balls: open $B(x,r) = \{ y\in X\colon \| y - x \| < r \}$, closed $\ovline{B}(x,r) = \{ y\in X\colon \| y - x \| \lesq r \}$;
the (closed) sphere $S(x,r) = \{ y\in X\colon \| y - x \| = r \}$. $r > 0$ is assumed.

We call the sets $A\subseteq X$ and $B\subseteq X$ {\it congruent}, $A\cong B$, iff there is an isometric surjective mapping ({\it motion})
$f\colon X \lrarr X$: $\forall x, y \in X$: $\| f(x) - f(y) \| = \| x - y \|$ (surjectivity implies that $f^{-1} \colon X \lrarr X$ is a motion too), and $f(A) = B$.
The identity map $\mathcal{I}$: $\mathcal{I}(x) = x$ is a motion.

$\intr A = \{ x\in A\mid \exists \veps > 0\colon B(x,\veps) \subseteq A \}$ and
$\ovline{A} = \{ x \in X \mid \forall \veps > 0 \colon B(x, \veps) \cap A \ne \nullset \}$ are the interior and the closure of $A$, respectively.

{\ }

We assume that $X$ has these additional properties:

$\bullet$ $\dim X > 1$: $\exists a, b \in X$, which are linearly independent.

$\bullet$ NCS: $\| \cdot \|$ is strictly convex, that is, $\forall x, y \in S(\theta, 1)$, $x \ne y$: $\lambda \in (0;1)$ $\Rarr$ $\| \lambda x + (1 - \lambda ) y\| < 1$.

Conventional examples of non-NCS space are $\mathbb{R}^m_1$ and $\mathbb{R}^m_{\infty}$ when $m \greq 2$, or $L_1$ and $L_{\infty}$.

\section{Preliminaries}

\forceparindent
Watery Warning: some of the following lemmas seem ``folkloric'', with proofs included for the sake of integrity and probably present elsewhere.

\begin{lemma}\label{lmMotionSph}
If $f\colon X \lrarr X$ is a motion, then $\forall S(x,r)$: $f\bigl( S(x,r)  \bigr) = S(f(x), r)$.
\end{lemma}

\begin{proof}
a) $\forall y \in S(x,r)$: $\| f(y) - f(x) \| = \| y - x \| = r$ $\Rarr$ $f\bigl(S(x,r)\bigr) \subseteq S(f(x),r)$.
b) $\forall z\in S(f(x),r)$: $z = f(y)$, $\| y - x \| = \| f(y) - f(x) \| = \| z - f(x) \| = r$ $\Rarr$ $y \in S(x,r)$ $\Rarr$ $S(f(x), r) \subseteq f\bigl(S(x,r)\bigr)$.
\end{proof}

\begin{lemma}\label{lmMotionBall}
If $f\colon X \lrarr X$ is a motion, then $\forall B(x,r)$: $f\bigl( B(x,r)  \bigr) = B(f(x), r)$.
\end{lemma}

\begin{proof}
$f\bigl( B(x,r) \bigr) = f\bigl( \{ x \} \cup \bigcup\limits_{u\in (0;r)} S(x,u) \bigr) = \{ f(x) \} \cup \bigcup\limits_{u\in (0;r)} f\bigl( S(x,u) \bigr)
\stackrel{\text{Lemma \ref{lmMotionSph}}}{=}$

\hfill $= \{ f(x) \} \cup \bigcup\limits_{u\in (0;r)} S\bigl( f(x), u \bigr) = B\bigl( f(x), r \bigr)$
\end{proof}

\begin{lemma}\label{lmMotionDecomp}
Let $f\colon X \lrarr X$ be a motion. Then $f = h \circ g$ (that is, $f(x) = h\bigl(g(x)\bigr)$),
where $g\colon X \lrarr X$ and $h\colon X \lrarr X$ are uniquely determined motions such that

1) $\forall x \in X$: $\| g(x) \| = \| x \|$ ($\LRarr$ $g(\theta) = \theta$);

2) $\exists a \in X$: $\forall x \in X$: $h(x) = x + a$.
\end{lemma}

\begin{proof}
Consider $g(x) = f(x) - f(\theta)$ and $h(x) = x + f(\theta)$. Obviously, $h \circ g = f$.

$\| g(x) \| = \| f(x) - f(\theta) \| = \| x - \theta \| = \| x \|$. (Implied by $g(\theta) {=} \theta$: $\| g(x) \| {=} \| g(x) {-} g(\theta) \| {=} \| x {-} \theta \|$.)

$g$ and $h$ are motions: $\| g(x) - g(y) \| = \| f(x) - f(y) \| = \| x - y \|$ and $\| h(x) - h(y) \| = \| x - y \|$ (isometry),
inverse maps $g^{-1}(x) = f^{-1}(x + f(\theta))$ and $h^{-1}(x) = x - f(\theta)$ imply surjectivity.

Uniqueness: $f(\theta) = h(g(\theta)) = h(\theta) = a$, $g(x) = h^{-1}(f(x)) = f(x) - a = f(x) - f(\theta)$.
\end{proof}

Here, we call $h$ {\it shift} and $g$ {\it non-shift} components of the motion $f$. If $h=\mathcal{I}$ or $g=\mathcal{I}$, the respective component is called {\it trivial}.
It is easy to see that if $f$ has trivial shift or non-shift component, then the respective component of $f^{-1} = g^{-1} \circ h^{-1}$ is trivial as well.

{\ }

\begin{theorem}\label{thmIsomZero2Zero} (Mazur-Ulam, \cite{mazur1932}; \cite[5.3, Th. 12]{lax2002}).
The motion that maps $\theta$ to $\theta$ is linear.
\end{theorem}

{\bf Remark.} We consider the isometries that map $X$ onto itself,
while the theorem holds true for any bijective isometry between two normed spaces $X$ (with $\theta_X$) and $Y$ (with $\theta_Y$).

{\bf Corollary.} Non-shift component $g$ of the motion $f$ is linear: $g(\lambda x + \mu y) = \lambda g(x) + \mu g(y)$.

{\ }

\begin{lemma}\label{lmDiamTrivShift}
If the motion $f\colon X \lrarr X$ is such that $\exists x \in X$: $\| f(x) \| \lesq \| x \|$ and $\| f(-x) \| \lesq \| x \|$, then the shift component of $f$ is trivial.
\end{lemma}

\begin{proof}
Using the notation of Lemma \ref{lmMotionDecomp}, let $f = h \circ g$ and $y = g(x)$.
For $x = \theta$: $y = \theta$, so $f(x) = a$, and $\| a \| \lesq 0$ $\LRarr$ $a = \theta$. Suppose $x \ne \theta$.

By Th. \ref{thmIsomZero2Zero}, $-y = g(-x)$, so $f(x) = y + a$ and $\| y + a \| \lesq \| x \| = \| y \|$,
$f(-x) = -y + a$ and $\| -y + a\| \lesq \| y \|$ $\LRarr$ $\| y - a \| \lesq \| y \|$.
If $\| y + a\| < \| y \|$ or $\| y - a \| < \| y \|$, then by triangle inequality $2 \| y \| = \| y - (-y) \| \lesq \| y - a \| + \| a - (-y) \| < 2 \| y \|$, --- a contradiction;
thus $\| y - a \| = \| y + a \| = \| y \|$.

$y = \frac{1}{2}(y-a) + \frac{1}{2}(y+a)$. Assume $a \ne \theta$ $\LRarr$ $y - a \ne y + a$. For $s = (y-a) / \| y \|$, $t = y / \| y \|$, $u = (y+a) / \| y \|$:
$s,t,u \in S(\theta, 1)$, $s \ne u$, $\| \frac{1}{2}s + \frac{1}{2}u \| = \| t \| = 1$, contradicting NCS. So $a = \theta$.
\end{proof}

{\ }

Let $a_1$, ..., $a_m$ be linear independent (LI) elements of $X$ (thus $\dim X \greq m$).
We denote by $M(a_1, ..., a_m) = \bigl\{ \sum\limits_{i=1}^m x_i a_i \mid x_i \in \mathbb{R} \bigr\}$ the $m$-dimensional linear manifold generated by them.
It follows from LI that $\forall x \in M(a_1, ..., a_m)$ the coordinates $\{ x_i \}$ are determined uniquely.
Suppose $x^{(k)}, y \in M(a_1, ..., a_m)$. Since $\| x^{(k)} - y \| \lesq \sum\limits_{i=1}^m |x^{(k)}_i - y_i| \cdot \| a_i \|$ by triangle inequality,
we immediately see that $x^{(k)}_i \xrarr[k\rarr \infty]{} y_i$ for $i=\ovline{1,m}$
implies $x^{(k)} \xrarr[k\rarr \infty]{} y$, that is, $\| x^{(k)} - y \| \xrarr[k\rarr \infty]{} 0$.

The converse implication and the closedness of $M(a_1, ..., a_m)$ (making it a subspace of $X$),
though known well enough (see \cite[1.2.3]{cotlar1974}, \cite[5.2, Ex. 4]{lax2002}), are obtained in the next lemma by ``elementary'' reasonings,
without resort to norm equivalence or functionals.

\begin{lemma}\label{lmFinDimManifClosed}
If $x^{(k)} \in M(a_1, ..., a_m)$ and $x^{(k)} \xrarr[k\rarr \infty]{} x \in X$,
then $x \in M(a_1, ..., a_m)$, which is closed therefore, and $x^{(k)}_i \xrarr[k\rarr \infty]{} x_i$ for $i=\ovline{1,m}$.
\end{lemma}

\begin{proof}
The proof is by induction over $\dim M(a_1, ..., a_m)$.

Let $m = 1$. The sequence $\{ x^{(k)} \}_k = \{ x^{(k)}_1 a_1 \}_k$ is convergent (conv.), therefore it is fundamental (fund.)
Assume that $\{ x^{(k)}_1 \}_k$ is not conv., then it isn't fund. due to completeness of $\mathbb{R}$:

\centerline{$\exists \veps_0 > 0$: $\forall N \in \mathbb{N}$: $\exists k_1, k_2 > N$: $|x^{(k_1)}_1 - x^{(k_2)}_1| \greq \veps_0$}

But then $\| x^{(k_1)} - x^{(k_2)} \| = | x^{(k_1)}_1 - x^{(k_2)}_1 | \cdot \| a_1 \| \greq \veps_0 \| a_1 \| > 0$, which contradicts the fund. of $\{ x^{(k)} \}_k$.
Hence $\exists \lim\limits_{k\rarr \infty} x^{(k)}_1 = \wtilde{x}_1$. Let $\wtilde{x} = \wtilde{x}_1 a_1$.
$\| x^{(k)} - \wtilde{x} \| = |x^{(k)}_1 - \wtilde{x}_1| \cdot \| a_1 \| \xrarr[k\rarr \infty]{} 0$ $\Rarr$ $x^{(k)} \rarr \wtilde{x}$ as $k\rarr \infty$.
This means that $x = \wtilde{x} \in M(a_1)$ and $x^{(k)}_1 \rarr x_1$ as $k\rarr \infty$.

Now suppose that the statement holds true for $\dim M\bigl(\{ a_i \}\bigr) = 1, 2, ..., m-1$.
Consider the conv. $\{ x^{(k)} \}_k = \{ \sum\limits_{i=1}^m x^{(k)}_i a_i \}_k$, it is fund.
Take any $i_0 = \ovline{1,m}$, for instance $i_0 = m$. Assume that $\{ x^{(k)}_m \}$ isn't conv., then it isn't fund.,
$\exists \veps_0 > 0$: $\forall N \in \mathbb{N}$: $\exists k_1, k_2 > N$: $|x^{(k_1)}_m - x^{(k_2)}_m| \greq \veps_0$, and

\centerline{$\| x^{(k_1)} - x^{(k_2)} \| = \bigl\| \sum\limits_{i=1}^m (x^{(k_1)}_i - x^{(k_2)}_i) a_i \bigr\| =$}

\centerline{$= |x^{(k_1)}_m - x^{(k_2)}_m| \cdot \bigl\| a_m + \sum\limits_{i=1}^{m-1} \frac{x^{(k_1)}_i - x^{(k_2)}_i}{x^{(k_1)}_m - x^{(k_2)}_m} a_i \bigr\| =
|x^{(k_1)}_m - x^{(k_2)}_m| \cdot \| a_m - z \|$}

\noindent
where $z \in M (a_1, ..., a_{m-1}) = M_{m-1}$. It follows from $\dim M_{m-1} = m-1$ and the induction hypothesis that $M_{m-1}$ is closed.
$a_m \notin M_{m-1}$ due to LI, therefore $\| a_m - z \| \greq \rho (a_m, M_{m-1}) > 0$. And we obtain
$\| x^{(k_1)} - x^{(k_2)} \| \greq \veps_0 \rho (a_m, M_{m-1}) > 0$,
which contradicts the fund. of $\{ x^{(k)} \}_k$. Hence $\exists \lim\limits_{k\rarr \infty} x^{(k)}_m = \wtilde{x}_m$,
and similarly $\exists \lim\limits_{k\rarr \infty} x^{(k)}_i = \wtilde{x}_i$ for $i = \ovline{1,m-1}$. Let $\wtilde{x} = \sum\limits_{i=1}^m \wtilde{x}_i a_i$.

$\| x^{(k)} - \wtilde{x} \| \lesq \sum\limits_{i=1}^m |x^{(k)}_i - \wtilde{x}_i| \cdot \| a_i \| \xrarr[k\rarr \infty]{} 0$, so $x^{(k)} \xrarr[k\rarr \infty]{} \wtilde{x}$.
Consequently, $x = \wtilde{x} \in M(a_1, ..., a_m)$ and $x^{(k)}_i \xrarr[k\rarr \infty]{} x_i$ for $i = \ovline{1,m}$.
By induction principle, the statement is true for $\forall m \in \mathbb{N}$.
\end{proof}

{\ }

\begin{theorem}\label{thmLSB} (Lusternik-Schnirelmann-Borsuk (LSB), \cite[II.5]{lusternik1930}, \cite{borsuk1933}; see also \cite{matousek2008}).
Let the sphere $S_m = \bigl\{ x\in \mathbb{R}^m \colon \| x \|_m = \sqrt{\sum\limits_{j=1}^m x_j^2} = r \bigr\} = \bigcup\limits_{i=1}^m A_i$, where $A_i$ are closed.

Then $\exists i_0$, $\exists x\in S_m$: $\{ x, -x \} \subseteq A_{i_0}$, --- one of $A_i$ contains the pair of antipodal points of $S_m$.
\end{theorem}

{\ }

The immediate corollary of LSB theorem is this generalization for normed spaces:

\begin{lemma}\label{lmLSBNormSpace}
Let $\dim X \greq m \in \mathbb{N}$, that is, $\exists a_1, ..., a_m \in X$, which are linearly independent.

If $S(\theta, r) = \bigcup\limits_{i=1}^m A_i$, where $A_i$ are closed, then $\exists A_{i_0}$, $\exists x \in S(\theta, r)$: $\{ x, -x \} \subseteq A_{i_0}$.
\end{lemma}

See \cite[p. 119]{bollobas2006}, and most likely it's mentioned in \cite{steinlein1985}; more general form is in e.g. \cite[p. 39]{arandjelovic1999}.

\begin{proof}
Let $L = M(a_1, ..., a_m)$ be the subspace of $X$ generated by $\{ a_i \}$ (by Lemma \ref{lmFinDimManifClosed}, $L$ is closed),
$C = S(\theta, r) \cap L$, $S_m = \bigl\{ y \in \mathbb{R}^m \colon \| y \|_m = 1 \bigr\}$.
$\forall x\in L$ has the unique representation $x = (x_1; ...; x_m) = \sum\limits_{i=1}^m x_i a_i$.
Therefore the mapping $s \colon C \rarr S_m$: $s(x) = \bigl( x_1 / \| x \|_m ; ...; x_m / \| x \|_m \bigr)$ is well defined.
Moreover, we claim that $s$ is a homeomorphism.

1) $s$ is injective. Indeed, if $s(x') = s(x'')$, where $x', x'' \in C$, then $\frac{x'_i}{\| x' \|_m} = \frac{x''_i}{\| x'' \|_m}$ for $i = \ovline{1,m}$,
thus $x'_i = \alpha x''_i$ for $\alpha = \| x' \|_m / \|  x'' \|_m > 0$. So $x' = \alpha x''$ $\Rarr$ $r = \| x' \| = |\alpha| \cdot \| x'' \| = r \alpha$ $\Rarr$ $\alpha = 1$.

2) $s$ is surjective. $\forall y=(y_1;...;y_m) \in S_m$: $s^{-1}(y) = \frac{r}{\| x \|} x$, where $x = \sum\limits_{i=1}^m y_i a_i$.

3) $s$ is continuous. Let $C \ni x^{(k)} \xrarr[k\rarr \infty]{} x$. Using closedness of $S(\theta, r)$ and Lemma \ref{lmFinDimManifClosed}, we obtain:
$x \in S(\theta, r) \cap L = C$ and $x^{(k)}_i \xrarr[k\rarr \infty]{} x_i$.
Therefore $\| x^{(k)} \|_m \xrarr[k\rarr \infty]{} \| x \|_m$, and $s(x^{(k)}) \xrarr[k\rarr \infty]{} s(x)$.

4) $s^{-1}$ is continuous too. For $S_m \ni y^{(k)} \xrarr[k\rarr \infty]{} y$: $S_m$ is closed $\Rarr$ $y \in S_m$, and $y^{(k)}_i \xrarr[k\rarr \infty]{} y_i$.
Let $x^{(k)} = \sum\limits_{i=1}^m y^{(k)}_i a_i$, $x = \sum\limits_{i=1}^m y_i a_i$, then $x^{(k)} \xrarr[k\rarr \infty]{} x$,
$\| x^{(k)} \| \xrarr[k\rarr \infty]{} \| x \|$, so $s^{-1}(y^{(k)}) \xrarr[k\rarr \infty]{} s^{-1}(y)$.

Consider $C_i = A_i \cap C = A_i \cap S(\theta, r) \cap L$, they are closed.
Hence the image $s(C_i) \subseteq S_m$, under homeomorphic mapping $s$, is closed too (\cite[XII, \S 3]{kuratowski1961}).
$\bigcup\limits_{i=1}^m C_i = \bigl( \bigcup\limits_{i=1}^m A_i \bigr) \cap C = S(\theta, r) \cap C = C$, so $\bigcup\limits_{i=1}^m s(C_i) = S_m$.
By LSB theorem, $\exists i_0$, $\exists y\in S_m$: $\{ y,-y \} \in s(C_{i_0})$.
Since $s^{-1}(-y) = - s^{-1}(y)$ (by (2)), we obtain $x = s^{-1}(y) \in C$: $\{ x, -x \} \in C_{i_0} \subseteq A_{i_0}$.
\end{proof}

{\ }

In the Main section, certain infinite-dimensional ball covering will be considered, where the following Hilbert space-related lemmas are needed.

{\ }

We denote by $H = l_2$ the separable infinite-dimensional Hilbert space over $\mathbb{R}$.
Until the end of this section, $\| \cdot \| = \| \cdot \|_H$ denotes the norm in $H$. $S = S(\theta, 1)$ is the unit sphere of $H$.

$\inprod{x}{y}$ is the scalar/inner product of $x, y\in H$,
$\angle(x, y) = \arccos \frac{\inprod{x}{y}}{\| x \| \cdot \| y \|} \in [0;\pi]$ is the angle between $x$ and $y$
($\angle(x,y) = 0$ if $x = \theta$ or $y = \theta$). $x \perp y$ means $\inprod{x}{y} = 0$.
The ``basic'' properties of $H$ and $\inprod{\cdot}{\cdot}$ (like $\inprod{x}{x} = \| x \|^2$) are assumed to be known; see e.g. \cite[II.3]{cotlar1974}, \cite[6]{lax2002}.

\begin{lemma}\label{lmCountDenseSphGeod}
$\exists D = \{ d_i \}_{i\in \mathbb{N}} \subset S$ such that $\forall \beta > 0$, $\forall x \in S$: $\exists d \in D$: $\angle(x,d) < \beta$.
\end{lemma}

In other words, there's a countable subset $D$ of $S$, which is everywhere dense (ED) in ``geodesic'' metric $\rho_S(x,y) = \angle(x,y)$ on $S$
(see \cite[6.4, 17.4]{deza2009}). Such $D$ is said to be {\it geodesically dense in $S$}.

\begin{proof}
$H$ is separable: $\exists R \subset H$, countable and ED in $H$. Let $D = \{ \frac{y}{\| y \|} \mid y\in R, y\ne \theta \}$; $D \subset S$ and $D$ is countable.
Take any $x \in S$. $\forall \delta > 0$ $\exists y \in R$: $\| x - y \| < \delta$. Then, by triangle inequality,

\centerline{$ 1 - \delta < \| x \| - \| x - y \| \lesq \| y \| \lesq \| y - x \| + \| x \| < 1 + \delta$}

\noindent
hence $\| x - \frac{y}{\| y \|} \| = \frac{1}{\| y \|} \cdot \bigl\| x \cdot \| y \| - y \bigr\| \lesq \frac{1}{\| y \|} \Bigl[\bigl\| (\| y \| - 1) x  \bigr\|  + \| x - y \| \Bigr] \lesq
\frac{1}{1 - \delta} \bigl[ \delta \| x \| + \delta \bigr] = \frac{2\delta}{1 - \delta}$.

Since $\frac{2\delta}{1 - \delta} \rarr 0$ as $\delta \rarr 0$, we obtain for $\forall \veps > 0$: $\exists d = \frac{y}{\| y \|} \in D$: $\| x - d \| < \veps$.
Consider $\veps = \frac{1}{n}$ to get $\{ d_n \}_{n\in \mathbb{N}} \subset D$: $d_n \xrarr[n\rarr \infty]{} x$.
(So, $D$ is ED on $S$ in $\| \cdot \|$-induced metric).

It follows from continuity of $\| \cdot \|$ and $\inprod{\cdot}{\cdot}$ that
$\frac{\inprod{d_n}{x}}{\| d_n \| \cdot \| x \|} \xrarr[n\rarr \infty]{} \frac{\inprod{x}{x}}{\| x \|^2} = 1$.
In turn, continuous $\arccos \frac{\inprod{d_n}{x}}{\| d_n \| \cdot \| x \|} \xrarr[n\rarr \infty]{} 0$,
therefore $\exists d_{n_{\beta}} \in D$: $\angle(x,d_{n_{\beta}}) < \beta$.
\end{proof}

{\bf Remark.} Given such $D$, it is easy to see that $\{ A_i \}_{i\in \mathbb{N}} = \{ \ovline{B}(d_i, \veps) \cap S \}_{i\in \mathbb{N}}$ for $\veps < 1$
is a covering of $S$ by closed subsets (moreover, $A_i \cong A_j$, see Lemma \ref{lmOmtdCongr}).
$\dim H = \aleph_0 = |\{  A_i\}|$, however, $\diam A_i \lesq \diam \ovline{B}(d_i, \veps) = 2\veps < 2$, thus no $A_i$ contains antipodal points of $S$, ---
the ``straight'' attempt of infinite-dimensional generalization of LSB theorem fails. Cf. \cite{cutler1973}.

\begin{lemma}\label{lmHilbertNonShiftInprod}
If $g\colon H \lrarr H$ is a non-shift motion, $g(\theta) = \theta$, then $\forall x, y \in H$: $\inprodbig{g(x)}{g(y)} = \inprod{x}{y}$.
\end{lemma}

\begin{proof}
$\inprodbig{g(x)}{g(y)} = \inprodbig{g(x) - g(y) + g(y)}{g(y)} = \inprodbig{g(x) - g(y)}{g(y)} + \| g(y) \|^2 =$

\centerline{$= \inprodbig{g(x) - g(y)}{g(y) - g(x) + g(x)} + \| y \|^2 \stackrel{\text{Th. \ref{thmIsomZero2Zero}}}{=}$}

\noindent
$= - \inprodbig{g(x - y)}{g(x - y)} + \| g(x) \|^2 - \inprodbig{g(x)}{g(y)} + \| y \|^2 =
\| x \|^2 + \| y \|^2 - \| x - y \|^2 - \inprodbig{g(x)}{g(y)}$,
therefore $\inprodbig{g(x)}{g(y)} = \frac{1}{2} \bigl[ \| x \|^2 + \| y \|^2 - \inprod{x - y}{x - y} \bigr] = \frac{1}{2} \bigl[ 2 \inprod{x}{y} \bigr] = \inprod{x}{y}$.
\end{proof}

{\ }

{\bf Definition.} Let $H \ni s \ne e \in H$, $\gamma \in [0;\pi]$. We call the set

\centerline{$C(s, e, \gamma) = \bigl\{ x \in H \colon \| x - s \| \lesq \| e - s \| \text{ and } \angle(x - s, e - s) \lesq \gamma \bigr\} \subseteq \ovline{B}(s, \| e - s \|)$}

\noindent
the (closed) {\it ommatidium}, with origin at $s$, around $[s, e]$, of angle $\gamma$ and of radius $\| e - s \|$.

It's actually a ``sector'' of the ball $\ovline{B}(s, \| e - s \|)$, and would be a usual disk sector in $\mathbb{R}^2$.

\begin{lemma}\label{lmOmtdSegm}
If $s \ne x \in C(s, e, \gamma)$, then $\forall \lambda \in [0; \frac{\| e - s\|}{\| x - s \|}]$: $s + \lambda (x - s) \in C(s, e, \gamma)$.
\end{lemma}

\begin{proof}
It follows simply from the definition.
\end{proof}

\begin{lemma}\label{lmOmtdCongr}
Two ommatidiums of the same angle and radius are congruent in $H$.
\end{lemma}

\begin{proof}
Evidently, a parallel shift $h(x) = x + a$ transforms $C(s, e, \gamma)$ onto $C(s + a, e + a, \gamma)$.
Thus we consider, without loss of generality, $C_1 = C(\theta, e_1, \gamma)$ and $C_2 = C(\theta, e_2, \gamma)$,
where $\| e_1 \| = \| e_2 \| = r$, $e_1 \ne e_2$. We are going to find the non-shift motion $g$ such that $g(C_1) = C_2$.

It suffices to obtain $g$ such that $g(e_1) = e_2$. Indeed, $\forall x \in C_1$ we have then $\| g(x) \| = \|  x \| \lesq r$ and
$\angle(g(x), e_2) = \arccos \frac{\inprod{g(x)}{g(e_1)}}{\| g(x) \| \cdot \| g(e_1) \|} \stackrel{\text{Lemma \ref{lmHilbertNonShiftInprod}}}{=}
\arccos \frac{\inprod{x}{e_1}}{\| x \| \cdot \| e_1 \|} = \angle(x, e_1) \lesq \gamma$, so $g(x) \in C_2$.
Conversely, $\forall x \in C_2$: $g^{-1} (x) \in C_1$, because $g^{-1}$ is a non-shift motion as well, and $g^{-1} (e_2) = e_1$.

We apply the ``coordinate'' approach to define such $g$.

Let $e_1' = e_1 / r$, $e_2' = e_2 / r$. They generate the 2-dimensional subspace $M = M(e_1', e_2')$ of $H$.
$\exists u \in M$ such that $\| u \| = 1$ and $u \perp e_1'$, hence $M = M(e_1', u)$ and $\forall z\in M$: $z = z_1 e_1' + z_2 u$, $\| z \|^2 = z_1^2 + z_2^2$.
Then $e_2' = (\cos \alpha) e_1' + (\sin \alpha) u$ for some $\alpha \in (0;2\pi)$.

$H = M \oplus L$, where $L$ is the orthogonal complement of $M$. It follows that $\forall x \in H$ has unique representation $x = x_1 e_1' + x_2 u + w_x$, where $w_x \in L$,
and $\| x \|^2 = x_1^2 + x_2^2 + \| w_x \|^2$. In particular, $e_1 = r e_1'$ and $e_2 = (r \cos \alpha) e_1' + (r \sin \alpha) u$.

Let $g(x) = (x_1 \cos \alpha - x_2 \sin \alpha) e_1' + (x_1 \sin \alpha + x_2 \cos \alpha) u + w_x$. It has the required properties:

1) $g$ is isometric: $\| g(x) - g(y) \|^2 =$

\centerline{$= \bigl[ (x_1 - y_1) \cos \alpha - (x_2 - y_2) \sin \alpha \bigr]^2 + \bigl[ (x_1 - y_1) \sin \alpha + (x_2 - y_2) \cos \alpha \bigr]^2 + \| w_x - w_y \|^2 =$}

\centerline{$ = (x_1 - y_1)^2 + (x_2 - y_2)^2 + \| w_x - w_y \|^2 = \| x - y \|^2$}

2) $g$ is surjective: $g^{-1}(x) = (x_1 \cos \alpha + x_2 \sin \alpha) e_1' + (-x_1 \sin \alpha + x_2 \cos \alpha) u + w_x$.

3) $g(\theta) = \theta + \theta + \theta = \theta$, and 4) $g(e_1) = (r \cos \alpha) e_1' + (r \sin \alpha) u + \theta = e_2$.
\end{proof}

\begin{lemma}\label{lmOmtdCoverBall}
If $D = \{ d_i \}_{i\in \mathbb{N}} \subset S$ is geodesically dense in $S$, then $\forall \beta > 0$:
$\ovline{B}(\theta, 1) = \bigcup\limits_{i \in \mathbb{N}} C(\theta, d_i, \beta)$.
\end{lemma}

\begin{proof}
$C(\theta, d_i, \beta) \subseteq \ovline{B}(\theta, 1)$ is obvious.
$\theta \in C(\theta, d_i, \beta)$ for any $i$. Take $\forall x \in \ovline{B}(\theta, 1) \bsl \{ \theta \}$, then $x' = x / \| x \| \in S$ and $x = \| x \| \cdot x'$.
By definition of $D$, $\exists d \in D$: $\angle(x', d) < \beta$, thus $x' \in C(\theta, d, \beta)$.
By Lemma \ref{lmOmtdSegm}, $x \in C(\theta, d, \beta)$ too.
\end{proof}

\begin{lemma}\label{lmOmtdInsideBall}
Let $d \in S$ and $\gamma \lesq \arccos \frac{1}{4}$. Then $C_0 = C(-\frac{1}{2}d, \frac{1}{2}d, \gamma) \subset \ovline{B}(\theta, 1)$.
\end{lemma}

\begin{proof}
Due to convexity of $\ovline{B}(\theta, 1)$, we only need to prove that $\forall x \in S(-\frac{1}{2}d, 1) \cap C_0$: $\| x \| \lesq 1$,
because $\forall y \in C_0 \bsl \{ -\frac{1}{2}d \}$: $y = \lambda x + (1 - \lambda) (-\frac{1}{2} d)$
for $x = -\frac{1}{2}d + \frac{y + \frac{1}{2}d}{\|y + \frac{1}{2}d\|} \in S(-\frac{1}{2}d, 1) \cap C_0$ ($\in C_0$ follows from Lemma \ref{lmOmtdSegm})
and $\lambda = \| y + \frac{1}{2} d\| \in [0;1]$.

Let $H = M(d) \oplus T$, then $\forall y \in H$:
$y = y_1 d + y_2 u$, where $d \perp u \in T$, $\| u \| = 1$, and $\| y \|^2 = y_1^2 + y_2^2$.
In particular, for $y = x + \frac{1}{2}d$: $\| y \| = 1$, hence we can represent $y_1 = \cos \beta \greq 0$, $y_2 = \sin \beta$ for $\beta \in [0;2\pi)$.
At that $y_1 = \inprod{y}{d} = \frac{\inprod{y}{d}}{\| y \| \cdot \| d \|} = \cos \angle(x + \frac{1}{2}d,d)$.
$x \in C_0$, so $y_1 = \cos \beta \greq \cos \gamma \greq \frac{1}{4}$,

\hfill $\| x \|^2 = \| y - \frac{1}{2} d \|^2 = \| (\cos \beta - \frac{1}{2}) d + \sin \beta \cdot u \|^2 = (\cos \beta - \frac{1}{2})^2 + \sin^2 \beta =
\frac{5}{4} - \cos \beta \lesq 1$
\end{proof}

\begin{lemma}\label{lmOmtdOriginNonIntr}
If $\gamma < \pi$, then $s \notin \intr C(s, e, \gamma)$.
\end{lemma}

\begin{proof}
Let $v = e - s$, then $\forall \veps > 0$: $\angle (-\veps v, e - s) = \arccos \frac{\inprod{-\veps v}{v}}{\| -\veps v \| \cdot \| v \|} = \arccos (-1) = \pi > \gamma$,
hence $B(s, 2 \veps \| e - s\|) \ni s -\veps v \notin C(s, e, \gamma)$, and $B(s, 2 \veps \| e - s\|) \nsubseteq C(s, e, \gamma)$.
\end{proof}

\begin{lemma}\label{lmOmtdMiddleInt}
If $\gamma > 0$, then $\frac{1}{2}(s + e) \in \intr C(s, e, \gamma)$.
\end{lemma}

\begin{proof}
Without loss of generality, assume that $s = \theta$, $\| e \| = 1$, and $\gamma \lesq \frac{\pi}{4}$
(otherwise move the ommatidium so that its origin becomes $\theta$ by Lemma \ref{lmOmtdCongr}, scale it to attain $\| e \| = 1$ ($x\lrarr x / \| e \|$),
and consider $C(\theta, e, \frac{\pi}{4}) \subseteq C(\theta, e, \gamma)$).
We need to show that $\exists \veps > 0$: $B(\frac{1}{2}e, \veps) \subseteq C(\theta, e, \gamma)$ $\LRarr$ $\forall x\in B(\frac{1}{2}e, \veps)$:
$\| x\| \lesq 1$ and $\angle (x, e) \lesq \gamma$; the latter inequality is equivalent to $\cos \angle(x, e) \greq \cos \gamma$.

For arbitrary $\veps > 0$ and $\forall x \in B(\frac{1}{2}e, \veps)$: $x = \frac{1}{2}e + b$, where $\| b \| < \veps$.

Then $\| x\| \lesq \| \frac{1}{2} e\| + \| b\| < \frac{1}{2} + \veps$; the constraint $\veps < \frac{1}{2}$ ensures $\| x\| < 1$.

\centerline{$\cos \angle(x, e) = \frac{\inprod{x}{e}}{\| x \| \cdot \| e\|} = \frac{1}{\| \frac{1}{2}e + b \|} \bigl[ \inprod{\frac{1}{2}e}{e} + \inprod{b}{e} \bigr] =
\frac{1}{\| e + 2b\|} + \frac{2}{\| e + 2b\|} \inprod{b}{e}$}

1) $\| e + 2b \| \lesq \| e \| + 2 \| b \| < 1 + 2\veps$, hence $\frac{1}{\| e + 2b\|} > \frac{1}{1 + 2 \veps}$.

2) On the other hand, $\| e + 2b\| \greq \| e \| - \| - 2b\| > 1 - 2 \veps$ $\Rarr$ $\frac{2}{\| e + 2b\|} < \frac{2}{1 - 2\veps}$,
and Cauchy-Bunyakowsky-Schwartz inequality implies $\bigl| \inprod{b}{e} \bigr| \lesq \| b \| \cdot \| e \| < \veps$,
therefore $\frac{2}{\| e + 2b\|} \inprod{b}{e} > - \frac{2\veps}{1 - 2\veps}$.

Consequently (for sufficiently small $\veps$) $\cos \angle(x, e) > \frac{1}{1 + 2 \veps} - \frac{2\veps}{1 - 2 \veps} \rarr 1$ as $\veps \rarr 0$,
thus for some $\veps_0 \in (0;\frac{1}{2})$ we obtain: $\cos \angle(x, e) \greq \cos \gamma$ for each $x\in B(\frac{1}{2}e, \veps_0)$.
\end{proof}

\begin{lemma}\label{lmOmtdConvex}
If $\gamma \lesq \frac{\pi}{2}$, then $C(s, e, \gamma)$ is convex.
\end{lemma}

\begin{proof}
Again, we assume $s = \theta$ and $\| e \| = 1$ without loss of generality.

Let $x, y \in C(\theta, e, \gamma)$. We claim that $\forall \lambda \in [0;1]$: $z = \lambda x + (1 - \lambda) y \in C(\theta, e, \gamma)$.

If $x = \theta$ or $y = \theta$, then $z \in C(\theta, e, \gamma)$ by Lemma \ref{lmOmtdSegm}.
If not: clearly $\theta \in C(\theta, e, \gamma)$, suppose $z \ne \theta$.

1) $\| z\| \lesq \lambda \| x \| + (1 - \lambda) \| y \| \lesq \lambda \cdot 1 + (1 - \lambda) \cdot 1 = 1$;

2) $\cos \angle(z, e) = \frac{\inprod{z}{e}}{\| z \| \cdot \| e\|} = \bigl[ \lambda \frac{\inprod{x}{e}}{\| z \|} + (1 - \lambda) \frac{\inprod{y}{e}}{\| z \|} \bigr] =
\bigl[ \lambda \frac{\inprod{x}{e}}{\| x\|} \cdot \frac{\| x \|}{\| z \|} + (1 - \lambda) \frac{\inprod{y}{e}}{\| y\|} \cdot \frac{\| y \|}{\| z \|} \bigr] =$

\hfill $= \bigl[ \frac{\lambda \| x \|}{\| z\|} \cos \angle(x, e) + \frac{(1 - \lambda) \| y\|}{\| z \|} \cos \angle(y, e) \bigr] \greq
\frac{\lambda \| x\| + (1 - \lambda) \| y\|}{\| \lambda x + (1 - \lambda) y \|} \cos \gamma \stackrel{\cos \gamma \greq 0}{\greq} \cos \gamma$
$\LRarr$ $\angle(z, e) \lesq \gamma$
\end{proof}

\section{Main}

\begin{prop}\label{propMain}
Let $\dim X \greq m \in \mathbb{N}$, $B(\theta, 1) \subseteq E \subseteq \ovline{B}(\theta, 1)$, and $E = \bigcup\limits_{i=1}^m A_i$, where $A_i \cong A_j$.

Then either $\theta \in \bigcap\limits_{i=1}^m \intr A_i$, or $\theta \notin \bigcup\limits_{i=1}^m \intr A_i$.
\end{prop}

\begin{proof}
Suppose $m \greq 2$. Let $K = \ovline{B}(\theta, 1)$, $S = S(\theta, 1)$.
$K = \ovline{B(\theta, 1)} \subseteq \ovline{E} \subseteq \ovline{\ovline{B}(\theta, 1)} = K$ $\LRarr$ $\ovline{E} = K$.

Let $f_{ij}$ be the motion transforming $A_i$ to $A_j$, so that $f_{ij}(A_i) = A_j$, and $f_{ji} = f^{-1}_{ij}$ ($f_{ii} = \mathcal{I}$).

Consider $S_i = \ovline{A_i} \cap S$. They are closed and $\bigcup\limits_{i=1}^m S_i = \bigl( \bigcup\limits_{i=1}^m \ovline{A_i} \bigr) \cap S =
\ovline{\bigcup\limits_{i=1}^m A_i} \cap S = K \cap S = S$. By Lemma \ref{lmLSBNormSpace}, $\exists S_k$, $\exists d \in S$: $\{ d, -d \} \subseteq S_k$.

Take any $i \ne k$. Let $A_k' = A_k \cup S_k \cup f^{-1}_{ki}(S_i)$ and $A_i' = A_i \cup S_i \cup f_{ki}(S_k)$.

1) $A_i' \subseteq K$. Indeed, a) $A_i \subseteq E \subseteq K$, b) $S_i \subseteq S \subset K$, c) $\forall x \in S_k \subseteq \ovline{A_k}$
$\exists \{ x_l \}_{l=1}^{\infty}$, $x_l \in A_k$: $x_l \xrarr[l\rarr \infty]{} x$, then continuous $f_{ki}(x_l) \xrarr[l\rarr \infty]{} f_{ki}(x)$.
$f_{ki}(x_l) \in A_i$, hence $f_{ki}(x) \in \ovline{A_i} \subseteq \ovline{E} = K$.

2) $f_{ki}(A_k') = f_{ki} (A_k) \cup f_{ki}(S_k) \cup f_{ki} \bigl( f^{-1}_{ki} (S_i) \bigr) = A_i \cup S_i \cup f_{ki}(S_k) = A_i'$.

Since $\{ d, -d \} \subseteq S_k \subseteq A_k'$, we obtain $f_{ki} \bigl( \{d, -d \} \bigr) \subseteq A_i' \subseteq K$,
so $\| f_{ki} (d) \| \lesq 1 = \| d \|$ and $\| f_{ki}(-d) \| \lesq \| d \|$.
By Lemma \ref{lmDiamTrivShift}, the shift component $h_{ki}$ of $f_{ki} = h_{ki} \circ g_{ki}$ is trivial. Then the shift component $h_{ik}$ of $f_{ik} = f^{-1}_{ki}$ is also trivial.

There are 2 possible cases: either $\exists i$: $\theta \in \intr A_i$, or $\forall i$: $\theta \notin \intr A_i$ $\LRarr$ $\theta \notin \bigcup\limits_{i=1}^m \intr A_i$.

Consider the former case, then $\exists B(\theta, \veps) \subseteq A_i$. Take $\forall j \ne i$.
By Lemma \ref{lmMotionBall}, $f_{ik} \bigl( B(\theta, \veps) \bigr) = B\bigl( f_{ik}(\theta), \veps \bigr) = B\bigl( g_{ik}(\theta), \veps \bigr) = B(\theta, \veps)$,
hence $B(\theta, \veps) \subseteq A_k$. Apply Lemma \ref{lmMotionBall} again:
$f_{kj} \bigl( B(\theta, \veps) \bigr) = B\bigl( f_{kj}(\theta), \veps \bigr) = B(\theta, \veps) \subseteq A_j$, and $\theta \in \intr A_j$.
Therefore $\theta \in \bigcap\limits_{i=1}^m \intr A_i$.
\end{proof}

{\bf Corollary.} If $\dim X \greq \aleph_0$, then the statement of Prop. \ref{propMain} holds true for $\forall m \in \mathbb{N}$:
a ball in such $X$ cannot be covered by any finite number of congruent subsets so that its centre belongs to the interiors of certain of them
and doesn't belong to the interiors of the others.

As for infinite coverings, see Ex. \ref{exmpUnivInfCover} and Ex. \ref{exmpHilbertCountCover} below.

{\ }

\begin{remark}\label{remIneqCounterEx}
One may ask why we do not generalize the approach from \cite{douwen1993} instead.

The reasonings there essentially make use of the inequality

\centerline{$\forall x, y, z \in X$: $\| x - y \|^2 + \| z \|^2 \lesq \| x \|^2 + \| y \|^2 + \| x - z \|^2 + \| y - z \|^2$}

\noindent
which is the implication of the inequality \cite[p. 184, (c)]{douwen1993} (for $p \larr \theta$, $q \larr z = \sigma_A(\theta)$),
established for Euclidean/pre-Hilbert $X$. Unfortunately, it is not true for arbitrary NCS $X$: consider $X = \mathbb{R}^2_{3/2}$ with
$\| x \| = \bigl\| (x_1;x_2) \bigr\|_{3/2} = \bigl( |x_1|^{3/2} + |x_2|^{3/2} \bigr)^{2/3}$
and let $x = (1;0)$, $y = (0;1)$, $z = (1;1)$. Then

\centerline{$\| x - y \|^2 + \| z \|^2 = 2 \cdot 2^{\frac{4}{3}} = 4 \cdot 2^{\frac{1}{3}} >
4 = 1 + 1 + 1 + 1 = \| x \|^2 + \| y \|^2 + \| x - z \|^2 + \| y - z \|^2$}

(Maybe some subtler form of the inequality would work.)
\end{remark}

{\ }

\begin{remark}
On the other hand, LSB theorem is applied here too, being a ``foundation stone'' of the inference;
another pebble is that the motions transforming the subsets onto each other don't include parallel shift, otherwise one of antipodal points moves outside of the ball.

Antipodal/``diametral'' points and the constraints they impose are exploited, --- without resort to LSB theorem, --- in \cite[\S 4]{edelstein1988},
where NCS Banach spaces are considered; see Rem. \ref{remPlaneNonNCSBanachNCS} below.
\end{remark}

{\ }

\begin{remark}
If we replace the condition ``$A_i \cong A_j$'' by ``$\intr A_i \cong \intr A_j$'', then ``$\theta \in \intr A_1$ and $\theta \notin \intr A_2$'' becomes possible, evidently;
for example, in $\mathbb{R}^2$ take $z\in K$: $\| z \| = \frac{1}{2}$, and

\centerline{$A_1 = B(\theta, \frac{1}{8}) \cup \{ (x;y) \in K\mid x\in \mathbb{Q} \text{ and } y\in \mathbb{Q} \}$,
$A_2 = B(z, \frac{1}{8}) \cup \{ (x;y) \in K\mid x\notin \mathbb{Q} \text{ or } y\notin \mathbb{Q} \}$}

\noindent
then $A_1 \cup A_2 = K$, $\intr A_1 = B(\theta, \frac{1}{8}) \ni \theta$, $\intr A_2 = B(z, \frac{1}{8})$, so $\intr A_1 \cong \intr A_2$, $A_1 \cap A_2 = \nullset$.

Same happens if we replace ``congruence'' by ``homotheticity'': take $A_1 = K$ and let $A_2$, $A_3$, ... be the balls of sufficiently small radius $\rho$ so that
all of them can be placed within $K$ and don't contain its centre (in other words, $A_i = \rho K + c_i$, $\rho < \| c_i \| < 1 - \rho$ for $i \greq 2$).
\end{remark}

{\ }

\begin{example}\label{exmpNonNCS}
Without NCS, the Prop. \ref{propMain} statement can become false.
Consider non-NCS $l_{\infty} = \bigl\{ x = (x_1;x_2;...)\colon \| x \|_{\infty} = \sup\limits_{i\in \mathbb{N}} |x_i| < \infty \bigr\}$
and its unit ball $\ovline{B}(\theta,1) = \bigl\{ x\in l_{\infty} \colon \sup\limits_i |x_i| \lesq 1 \bigr\}$. For any odd $n \greq 3$
the subsets $A_i = \bigl\{ x\in \ovline{B}(\theta, 1) \colon x_1 \in [-1+2\frac{i-1}{n};-1 + 2 \frac{i}{n}]\bigr\}$, $i=\ovline{1,n}$, are congruent
(motion $f_{ij}(x) = (x_1 + 2\frac{j-i}{n};x_2;x_3;...)$ transforms $A_i$ to $A_j$), $\theta \in B(\theta, \frac{1}{2n}) \subset \intr A_{1 + \lfloor\frac{n}{2}\rfloor}$,
while $\theta \notin A_i$ for $i \ne 1 + \lfloor\frac{n}{2}\rfloor$, and $A_1 \cup ... \cup A_n = \ovline{B}(\theta, 1)$.
Instead of $l_{\infty}$, we can take $\mathbb{R}^m_{\infty}$ if $m \greq n$.
Note that this decomposition of $\ovline{B}(\theta, 1)$ is a dissection and a covering (but not a partition).
\end{example}

{\ }

\begin{remark}\label{remPlaneNonNCSBanachNCS}
Consider $\mathbb{R}^2_p$, $\| x \|_{2;p} = \bigl( |x_1|^p + |x_2|^p \bigr)^{\frac{1}{p}}$.
For $p = 2$, usual Euclidean metric, the original dissection problem posed in \cite[C6]{croft1991} remains unsolved.
For $p = 1$ or $p = \infty$, --- non-NCS case, --- $\ovline{B}(\theta, 1)$ is a square
(sides being parallel to $Ox_i$ for $p = \infty$, rotated by $\frac{\pi}{4}$ for $p = 1$),
trivially dissectable into 3 (5, 7, ...) congruent rectangles such that the centre $\theta$ is within one of them.

It is shown in \cite[\S 2]{edelstein1988} that $\ovline{B}(\theta, 1)$ in non-NCS $c_0$, $C_{[0;1]}$ is partitionable into $n$ congruent subsets for $\forall n \lesq \aleph_0$,
while in NCS Banach $X$ there's no such partition if $2 \lesq n < \min\{ \dim X , \aleph_0 \} + 1$.
\end{remark}

{\ }

\begin{example}\label{exmpCoverDisk}
Obviously, as Fig. \ref{figCoverDiskGen} illustrates, the ball/disk in $\mathbb{R}^2$ can be covered by $n \greq 4$ congruent and convex subsets such that its centre belongs to the interior of exactly one set; moreover, the centre is at positive distance from other sets.

\begin{figure}[h]
\centerline{\begin{tabular}{ccc}
\includegraphics[width=3cm]{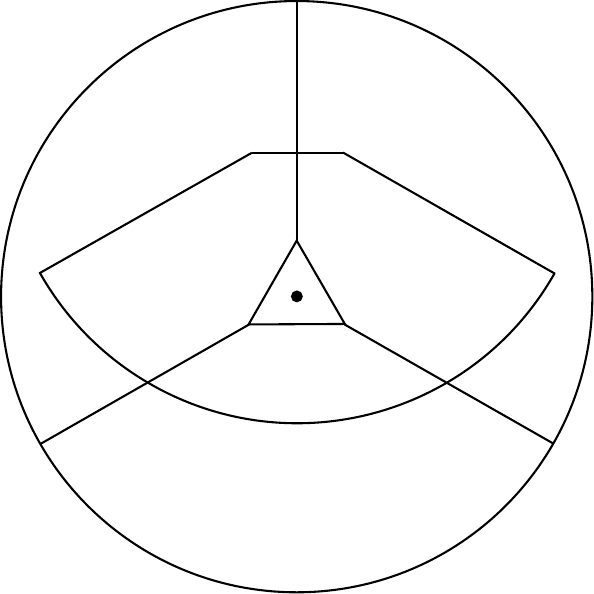}
&
\includegraphics[width=3cm]{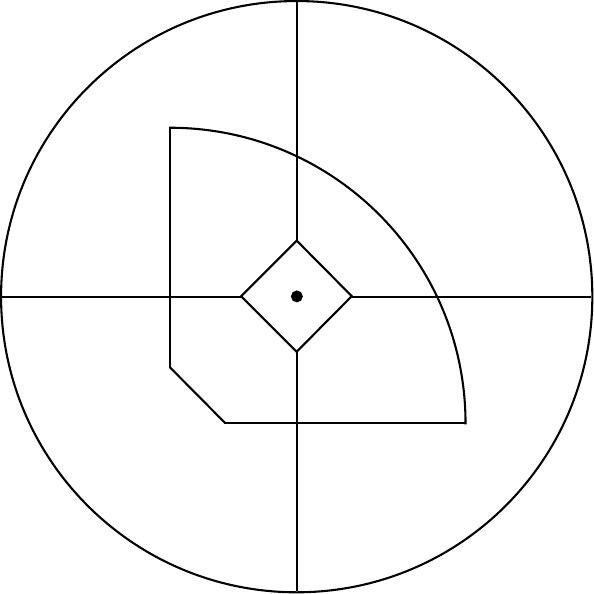}
&
\includegraphics[width=3cm]{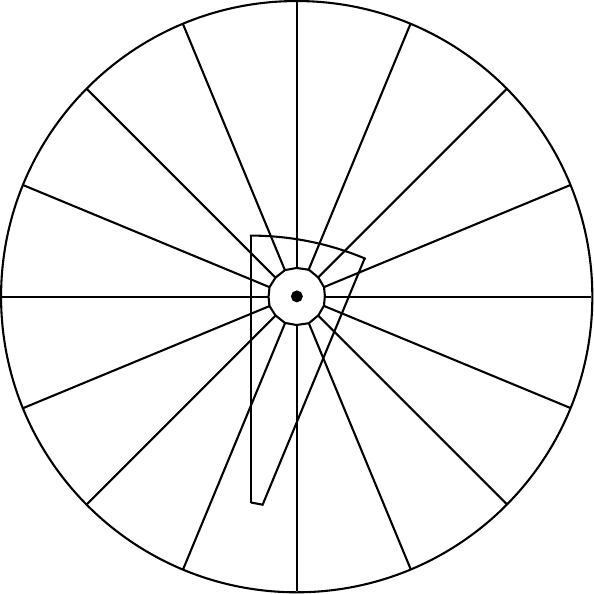}
\\
\includegraphics[height=1.5cm]{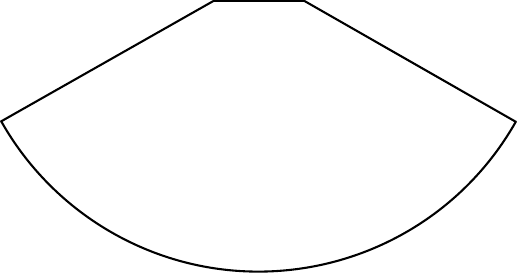}
&
\includegraphics[height=1.5cm]{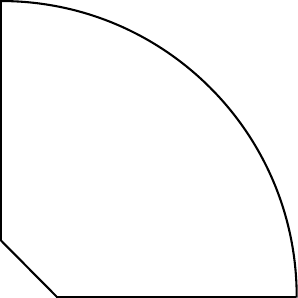}
&
\includegraphics[height=1.5cm]{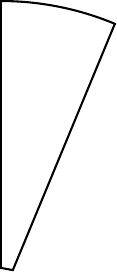}
\\
$n=4$ & $n=5$ & $n=17$
\end{tabular}}
\caption{Covering a disk by $n \greq 4$ congruent subsets}
\label{figCoverDiskGen}
\end{figure}

The case $n=3$ is slightly different: the sets are not convex and not 1-connected, each one has a circular hole in one of two symmetric segments it consists of.
At Fig. \ref{figCoverDisk3}, $\angle AOB = 150^{\circ}$ (for instance).
We do not know is there any such covering by three 1-connected congruent subsets.

\begin{figure}[h]
\centerline{
\begin{tabular}{ccc}
\includegraphics[width=3.6cm]{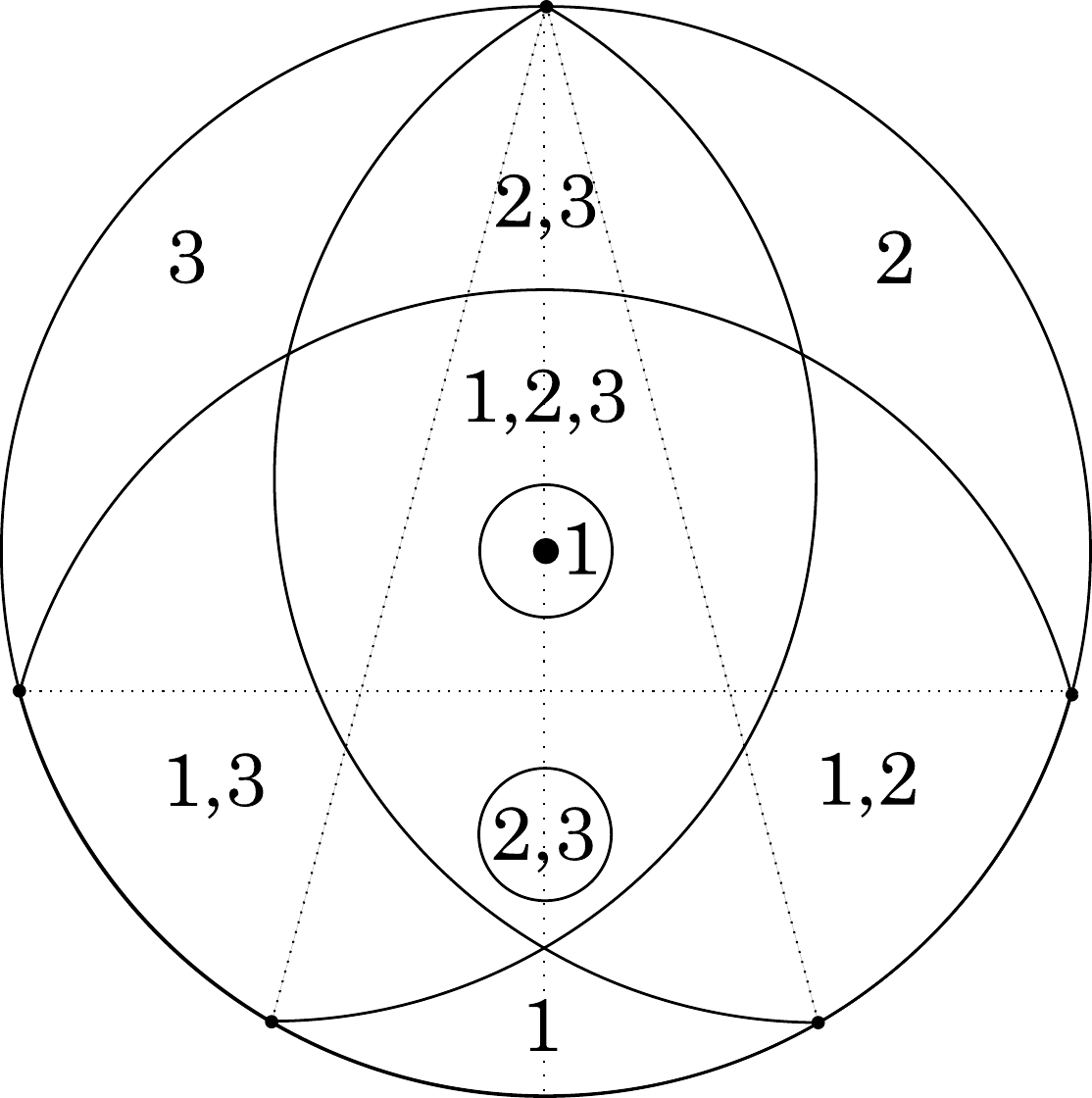}
&\quad\quad&
\includegraphics[width=3.5cm]{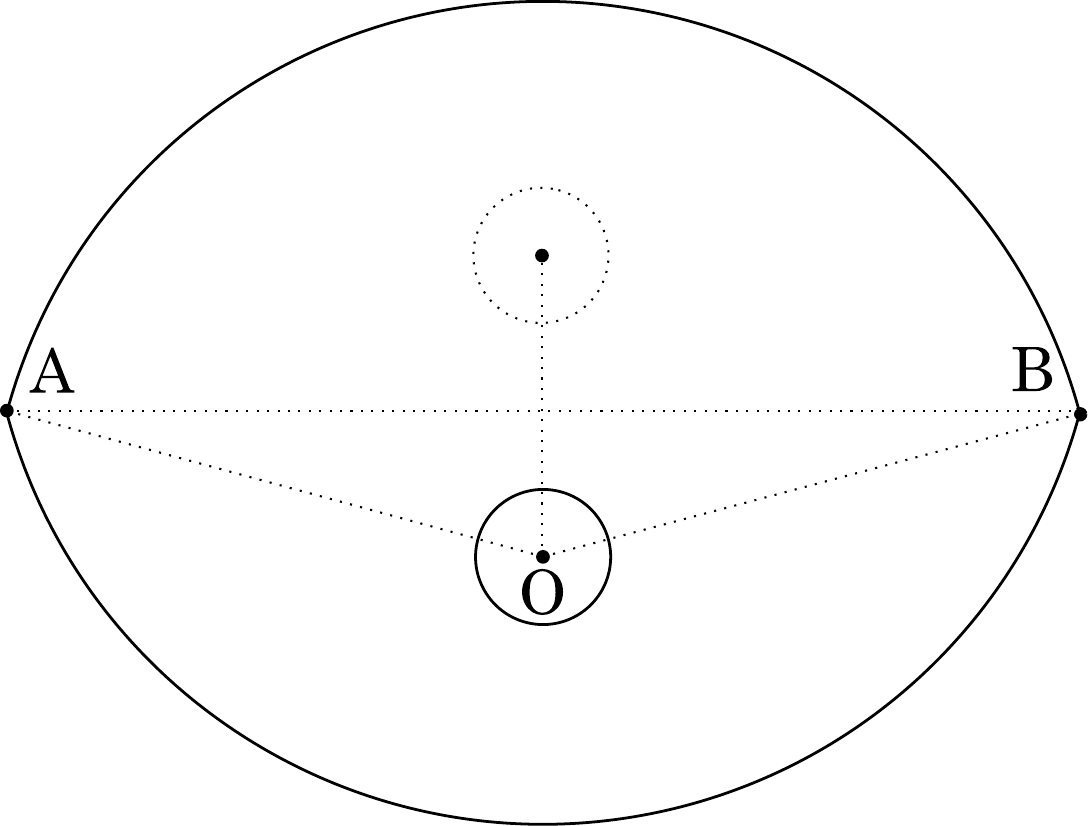}
\end{tabular}}
\caption{Covering a disk by $n=3$ congruent subsets (resembles ``Biohazard'' symbol)}
\label{figCoverDisk3}
\end{figure}

In fact, the case $n=k$ can be extended to all $n > k$ (which makes Fig. \ref{figCoverDiskGen} redundant),
because covering allows $A_i = A_j$ (not so for dissection): take $A_1$, ..., $A_{k-1}$, and $A_k = A_{k+1} = ... = A_n$.

Similar constructions can be used in $\mathbb{R}^m$. In particular, when $n = m + 2$,
note that the ``hollow'' around the centre at Fig. \ref{figCoverDiskGen}, case $n=4$, is an equilateral triangle and a 2-simplex in $\mathbb{R}^2$.
\end{example}

{\ }

\begin{remark}
Convexity of parts implies the negative answer not only to the original dissection problem, but also to its generalization:
the closed disk in $\mathbb{R}^2$ cannot be dissected into $n \greq 2$ homothetic, convex, and closed parts such that the interior of exactly one part contains the centre.

\begin{proof}
Let $K = \ovline{B}(\theta, 1)$, $S = S(\theta, 1)$ in $\mathbb{R}^2$, and let $\{A_i\}_{i=1}^n$ be the parts, so that $K=\bigcup\limits_{i=1}^n A_i$,
$A_i \sim A_j$, $\intr A_i \cap \intr A_j = \nullset$ for $i \ne j$.
Also, let $\partial A_i = \ovline{A_i} \cap \ovline{\mathbb{R}^2 \bsl A_i} \subset A_i$ be the boundary of $A_i$.

1) Claim: if $\partial A_i$ contains $2n+4$ different points $x_1$, ..., $x_{2n+4}$ that belong to some circle $S(a, r)$,
then $S(a,r) = S$.
(``The strictly convex section of $\partial A_i$ has to be on $\partial K = S$, not inside $K$.'')

To show that this claim is true, assume the contrary: $\exists x_j \notin S$. Let $N_+$ be the number of $x_j \in S$, and $N_- = \bigl|\{ x_j\colon x_j \notin S \}\bigr|$.
$N_+ + N_- = 2n + 4$.
If $N_+ \greq 3$, then 3 points that $\in S$ among $x_1$, ..., $x_{2n+4}$ determine the circle $S(a, r)$ uniquely (see e.g. \cite[2.3, Cor. 7]{agricola2008}),
so $S(a,r) = S$, which contradicts the assumption.
Thus $N_+ \lesq 2$ $\Rarr$ $N_- \greq 2n+2$: we can take $2n+2$ points on $S(a,r)$ in $\intr K = B(\theta, 1)$.
Enumerate them sequentially, for instance, counter-clockwise: $x_1'$, ..., $x_{2n+2}'$.

\begin{figure}[h]
\centerline{\includegraphics[width=3.5cm]{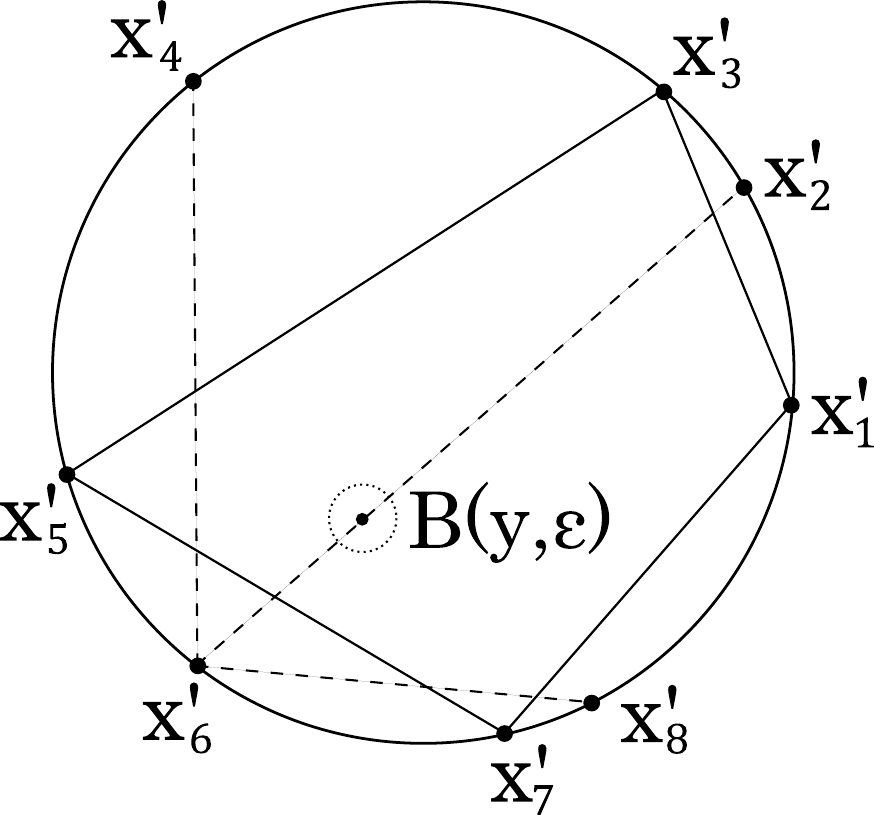}}
\label{figDissDiskConvex}
\end{figure}

Let $x_1' x_3' ... x_{2n+1}'$ be the convex polygon, with interior, inscribed into $S(a,r)$;
it follows from convexity of $A_i$ and $n+1\greq 3$ that $x_1' x_3' ... x_{2n+1}' \subseteq A_i$ and $\nullset \ne \intr x_1' x_3' ... x_{2n+1}' \subseteq \intr A_i$.
Consider the rest of points: $x_2'$, $x_4'$, ..., $x_{2n+2}'$.
Since $x_j' \in \partial A_i \cap \intr K$, each of these $n+1$ points belongs to the boundary $\partial A_{k_j}$ of at least one other part, $k_j \ne i$.
There are $n-1$ other parts, hence two of these points, $x_{(1)}' $ and $x_{(2)}'$, belong to the boundary of the same $A_k$, $k \ne i$.
Then $[x_{(1)}', x_{(2)}'] \subset A_k$.

It is easy to see that $[x_{(1)}', x_{(2)}']$ intersects $\intr x_1' x_3' ... x_{2n+1}'$, hence $\intr A_i \cap \intr A_k \ne \nullset$
(because $\forall y \in [x_{(1)}', x_{(2)}']$, $\forall B(y, \veps)$: $B(y,\veps) \cap \intr A_k \ne \nullset$), --- a contradiction.

2) $\bigcup\limits_{i=1}^n (A_i \cap S) = S$ and $|S| > \aleph_0$ imply $\exists A_i$: $\partial A_i \supseteq A_i \cap S \supseteq \{ x_1, ..., x_{2n+4} \}$.
Consequently, $\partial A_k = f_{ik}(\partial A_i)$ of any other $A_k$ contains $x^{(k)}_j = f_{ik}(x_j)$, $j=\ovline{1,2n+4}$,
which belong to some circle $S(a, r) = f_{ik}(S)$; here $f_{ik}\colon X \lrarr X$, $\| f_{ik} (x) - f_{ik}(y) \| = \alpha_{ik} \| x - y \|$ is the homothety transforming $A_i$ to $A_k$.
By (1), $x^{(k)}_j \in S$ for any $j$ and $k$.

3) Now assume that there's exactly one part $A_{i_0}$ such that $\theta \in \intr A_{i_0}$: $B(\theta, \delta) \subseteq A_{i_0}$.

$\{ x^{(i_0)}_j \}_{j=1}^{2n+4} \subseteq S \cap \partial A_{i_0}$, and the point $\theta \in A_{i_0}$ is at distance 1, equidistant, from each $x^{(i_0)}_j$.

Take any $k \ne i_0$. As above, $\{y^{(k)}_j = f_{i_0 k}(x^{(i_0)}_j)\}_{j=1}^{2n+4} \subseteq S \cap A_k$.
And there must be $z \in A_k$, which is equidistant from each $y^{(k)}_j$; clearly, $z = \theta$ ($y^{(k)}_1$, $y^{(k)}_2$, $y^{(k)}_3$ determine it uniquely).
Also, $B(\theta, \delta) \subseteq A_k$ (apply similar arguments to $\forall x \in B(\theta, \delta) \subseteq A_{i_0}$),
so $\theta \in \intr A_k$. A contradiction.
\end{proof}

(1st step shortens if we assume that one of $\partial A_i$ contains the arc $\breve{a}$, which has to be on $S$ then,
otherwise 2 internal points of $\breve{a}$ are in $\partial A_k$, $k\ne i$, too, implying a contradiction.)
\end{remark}

{\ }

\begin{example}\label{exmpUnivInfCover}
Without the upper bound for the cardinal number of covering,
there is a ``universal'' covering of $\ovline{B}(\theta, 1)$ such that the interior of exactly one subset contains the centre:
let $\mathcal{C} = \{ A_{\theta} \} \cup \bigl\{ A_y \bigr\}_{y \in S(\theta, \frac{1}{2})}$, where $A_{\theta} = \ovline{B}(\theta, \frac{1}{2})$ and
$A_y = \ovline{B}(y, \frac{1}{2})$.
Indeed, $\theta \in \intr A_{\theta} = B(\theta, \frac{1}{2})$, while $\forall y \in S(\theta, \frac{1}{2})$: $\theta \notin \intr A_y = B(y, \frac{1}{2})$,
and for $\forall x \in \ovline{B}(\theta, 1) \bsl \{ \theta \}$ we take $y_x = \frac{1}{2 \| x \|} x \in S(\theta, \frac{1}{2})$,
then $\| x - y_x \| = |1 - \frac{1}{2 \|  x\|}| \cdot \| x \| = \bigl| \| x \| - \frac{1}{2} \bigr| \lesq \frac{1}{2}$, thus $x \in A_{y_x}$.
Certainly, $A_i \cong A_j$.

This covering doesn't require NCS or $\dim X > 1$; meanwhile $\dim X > 1$ implies $| \mathcal{C} | > \aleph_0$.
\end{example}

{\ }

\begin{example}\label{exmpHilbertCountCover}
Consider the Hilbert space over $\mathbb{R}$, $X = H = l_2$, and its closed unit ball $\ovline{B_H} = \ovline{B}(\theta, 1)$, unit sphere $S_H = S(\theta, 1)$.
We claim that there is a countable covering of $\ovline{B_H}$ by its congruent and convex subsets $\{ A_i \}$ such that the interior of exactly one set contains the centre.

\begin{proof}
It's a direct corollary of Lemmas \ref{lmCountDenseSphGeod}, \ref{lmOmtdCongr}--\ref{lmOmtdConvex}:

1) Lemma \ref{lmCountDenseSphGeod} and Lemma \ref{lmOmtdCoverBall} along with Lemma \ref{lmOmtdCongr} provide
the countable covering of $\ovline{B_H}$ by congruent ommatidiums $A_i = C(\theta, d_i, \frac{\pi}{4})$, $i \in \mathbb{N}$, where $d_i \in S_H$.
By Lemma \ref{lmOmtdOriginNonIntr}, $\theta \notin \intr A_i$.

2) Then Lemma \ref{lmOmtdInsideBall} allows to add the ommatidium $A_0 = C(-\frac{1}{2}d_1, \frac{1}{2}d_1, \frac{\pi}{4})$,
which is contained in $\ovline{B_H}$ since $\frac{\pi}{4} < \arccos\frac{1}{4}$ and congruent with $A_i$ by Lemma \ref{lmOmtdCongr}.

3) Finally, by Lemma \ref{lmOmtdMiddleInt}, $\theta = \frac{1}{2}\bigl( -\frac{1}{2}d_1 + \frac{1}{2} d_1 \bigr) \in \intr A_0$.

By Lemma \ref{lmOmtdConvex}, $A_i$ are convex.

$|\{ A_i \}_{i \in \mathbb{N} \cup \{ 0\}}| = \aleph_0 = \dim H$.
\end{proof}

This covering somewhat resembles those from Fig. \ref{figCoverDiskGen}, except that
a)~the sets intersect ``a~lot'', b)~there's no ``hollow'' at the centre (corrigible by erasing sufficiently small neighborhood of template ommatidium's origin), and
c)~it's infinite-dimensional.
\end{example}

{\ }

The covering problem turns out to be easier about ``positive'' results than the problems of dissection and partition types.

\end{document}